\documentclass{amsart}
\usepackage{mathptmx, amscd, amssymb, colonequals, hyperref, url}
\DeclareMathAlphabet{\mathcal}{OMS}{cmsy}{m}{n}

\newtheorem{theorem}{Theorem}[section]
\newtheorem{lemma}[theorem]{Lemma}
\newtheorem{corollary}[theorem]{Corollary}
\newtheorem{proposition}[theorem]{Proposition}

\theoremstyle{definition}
\newtheorem{example}[theorem]{Example}
\newtheorem{remark}[theorem]{Remark}
\numberwithin{equation}{theorem}

\def\ge{\geqslant}
\def\le{\leqslant}
\def\to{\longrightarrow}
\def\mapsto{\longmapsto}
\def\into{\lhook\joinrel\longrightarrow}
\def\onto{\relbar\joinrel\twoheadrightarrow}

\def\codim{\operatorname{codim}}
\def\Proj{\operatorname{Proj}}
\def\Sing{\operatorname{Sing}}
\def\Spec{\operatorname{Spec}}
\def\Sym{\operatorname{Sym}}

\def\fraka{\mathfrak{a}}
\def\frakm{\mathfrak{m}}
\def\frakp{\mathfrak{p}}

\def\FF{\mathbb{F}}
\def\NN{\mathbb{N}}
\def\PP{\mathbb{P}}
\def\ZZ{\mathbb{Z}}

\def\calI{\mathcal{I}}
\def\calN{\mathcal{N}}
\def\calO{\mathcal{O}}

\begin{document}
\title[An asymptotic vanishing theorem for the cohomology of thickenings]{An asymptotic vanishing theorem for \\ the cohomology of thickenings}

\author[Bhatt]{Bhargav Bhatt}
\address{Department of Mathematics, University of Michigan, 530 Church Street, Ann Arbor, MI~48109, USA}
\email{bhargav.bhatt@gmail.com}

\author[Blickle]{Manuel Blickle}
\address{Institut f\"ur Mathematik, Fachbereich 08, Johannes Gutenberg-Universit\"at Mainz,
\newline 55099~Mainz, Germany}
\email{blicklem@uni-mainz.de}

\author[Lyubeznik]{Gennady Lyubeznik}
\address{Department of Mathematics, University of Minnesota, 206 Church~St., Minneapolis,\newline MN~55455, USA}
\email{gennady@math.umn.edu}

\author[Singh]{Anurag K. Singh}
\address{Department of Mathematics, University of Utah, 155 S 1400 E, Salt Lake City,\newline UT~84112, USA}
\email{singh@math.utah.edu}

\author[Zhang]{Wenliang Zhang}
\address{Department of Mathematics, Statistics, and Computer Science, University of Illinois at Chicago, 851 S.~Morgan~St., Chicago, IL 60607, USA}
\email{wlzhang@uic.edu}

\thanks{B.B.~was supported by NSF grant DMS~1801689, a Packard fellowship, and the Simons Foundation grant~622511, M.B.~by DFG grant SFB/TRR45, G.L.~by NSF grant DMS~1800355, A.K.S.~by NSF grant DMS~1801285, and W.Z.~by NSF grant DMS~1752081. The authors are grateful to the Institute for Advanced Study for support through the Summer Collaborators Program, to the American Institute of Mathematics for support through the SQuaREs Program, and to the referee for several helpful comments.}

\begin{abstract}
Let $X$ be a closed equidimensional local complete intersection subscheme of a smooth projective scheme $Y$ over a field, and let $X_t$ denote the~$t$-th thickening of $X$ in $Y$. Fix an ample line bundle $\mathcal{O}_Y(1)$ on $Y$. We prove the following asymptotic formulation of the Kodaira vanishing theorem: there exists an integer $c$, such that for all integers~$t \ge 1$, the cohomology group $H^k(X_t,\mathcal{O}_{X_t}(j))$ vanishes for $k < \dim X$ and~$j < -ct$. Note that there are no restrictions on the characteristic of the field, or on the singular locus of~$X$. We also construct examples illustrating that a linear bound is indeed the best possible, and that the constant $c$ is unbounded, even in a fixed dimension. 
\end{abstract}
\maketitle

\section{Introduction}
\label{section:introduction}

Let $Y$ be a projective scheme over a field, and let $X$ be a closed subscheme defined by an ideal sheaf~$\calI \subset\calO_Y$. For integers $t \ge 1$, let $X_t$ denote the $t$-th \emph{thickening} of $X$ in $Y$, i.e., the closed subscheme of $Y$ defined by $\calI^t$. In \cite{BBLSZ}, we proved the following version of the Kodaira vanishing theorem for thickenings of local complete intersection (lci) subvarieties of projective space $\PP^n$:

\begin{theorem}\cite[Theorem~1.4]{BBLSZ}
\label{theorem:BBLSZ1}
Let $X$ be a closed lci subvariety of $\PP^n$ over a field of characteristic zero. Then, for each $t\ge 1$ and $k < \codim(\Sing X)$, one has
\[
H^k(X_t,\ \calO_{X_t}(j)) = 0 \qquad\text{for } j < 0.
\]
\end{theorem}

When $X$ is smooth and $t=1$, this is precisely what is obtained from the Kodaira vanishing theorem. There are well-known counterexamples in the case of positive char\-acteristic~\cite{Raynaud, LR}; the condition on the singular locus is needed as well in view of the examples from~\cite{Arapura-Jaffe}. Nonetheless, as we prove here, there is an \emph{asymptotic} version of the above vanishing theorem that holds in good generality:

\begin{theorem}
\label{theorem:asymptotic}
Let $Y$ be a smooth projective scheme over a field, equipped with an ample line bundle $\calO_Y(1)$. Let $X$ be a closed equidimensional lci subscheme of $Y$. Then there exists an integer $c \ge 0$, such that for each $t \ge 1$ and $k < \dim X$, one has
\[
H^k(X_t,\ \calO_{X_t}(j)) = 0 \qquad \text{for all } j < -ct,
\]
where, for a closed subscheme $Z \subset Y$ and integer $j$, we write $\calO_Z(j)\colonequals\calO_Y(1)^{\otimes j}|_Z$.
\end{theorem}

Unlike Theorem~\ref{theorem:BBLSZ1} that relies on Hodge-theoretic input (via Kodaira vanishing), the proof of Theorem~\ref{theorem:asymptotic} only uses Serre vanishing; this is why we do not need any assumption on the characteristic of the field in Theorem~\ref{theorem:asymptotic}.

In the case where $Y=\PP^n$, with $\calO_Y(1)$ the standard ample line bundle, Theorem~\ref{theorem:asymptotic} answers \cite[Questions~7.1~and~7.2]{DM:TAMS} in the lci case; see Corollaries~\ref{corollary:lim1} and~\ref{corollary:lim2}. The linear bound in Theorem~\ref{theorem:asymptotic} is best possible in view of Example~\ref{example:linear} where, for each integer $c\ge2$, we construct an lci scheme $X$ of dimension $1$ such that, for each $t\ge 1$, the cohomology group $H^0(X_t,\ \calO_{X_t}(j))$ vanishes for~$j\le -ct$, and is nonzero for $j=-ct+1$. Theorem~\ref{theorem:asymptotic} may fail---even in characteristic zero---when $X$ is not lci, see Example~\ref{example:non:lci}, or when $X$ is lci but not equidimensional, see Example~\ref{example:non:equidimensional}.

\section{Preliminaries}
\label{sec:Prelim}

Let $X$ be a projective scheme over a field $\FF$. Set $d\colonequals\dim X$. We use $D_{coh}(X)$ to denote the derived category of complexes
\[
\CD
\cdots @>>> P^{i-1} @>>> P^i @>>> P^{i+1} @>>> \cdots
\endCD
\]
of $\calO_X$-modules with coherent cohomology, and~$D^b_{coh}(X)$ for the full triangulated subcategory of bounded complexes, i.e., those with only finitely many nonzero cohomology groups. We use $D^{\le a}_{coh}(X)$ (resp. $D^{\ge a}_{coh}(X)$) for complexes whose cohomology vanishes for~$i>a$ (resp. $i< a$). It is straightforward that each complex in $D^{\le a}_{coh}(X)$ (resp. $D^{\ge a}_{coh}(X))$ is quasi-isomorphic to a complex $P^\bullet $ such that $P^i=0$ for $i>a$ (resp. $i<a$). In particular, each complex in $D^b_{coh}(X)$ is quasi-isomorphic to a complex $P^\bullet$ such that $P^i\ne 0$ only for finitely many integers $i$.

We use $D^{\le a}(\FF)$ to denote the derived category of complexes of $\FF$-vector spaces whose cohomology vanishes for $i>a$, with $D^{\ge a}(\FF)$ defined analogously.

Since the global section functor $R\Gamma(X,-)$ sends a coherent sheaf $E$ on $X$ to a complex in~$D^{\le d}(\FF)$, and since each element $P$ in~$D^b_{coh}(X)\cap D^{\le a}_{coh}(X)$ is represented by a complex~$P^\bullet$ such that $P^i\ne 0$ only for finitely many $i$ and $P^i=0$ for $i>a$, it follows by applying the hypercohomology spectral sequence to $P^\bullet$ that the complex $R\Gamma(X,P^\bullet)$ lies in $D^{\le a + d}(\FF)$; while we do not need it here, this is true even without the boundedness assumption.

A key technical ingredient is the derived $m$-th divided power functor
\[
\Gamma^m\colon D^{\le 0}_{coh}(X) \to D^{\le 0}_{coh}(X)
\]
constructed in \cite{Illusie}, see also \cite[Chapter~25]{LurieSAG} or \cite{Quillen}. We summarize the properties of $\Gamma^m$ that we use in this paper. For a locally free sheaf $E$ of finite rank, $\Gamma^m$ is the usual $m$-th divided power of $E$. In particular, one has in this case,
\[
\Gamma^m(E)\ =\ \Sym^m(E^\vee)^\vee,
\]
where $(-)^\vee = \mathcal{H}om(-,\calO_X)$. By \cite[25.2.4.1]{LurieSAG}, the functor $\Gamma^m$ preserves $D^{\le a}_{coh}(X)$ for all integers $a \le 0$. Just as divided powers are not an additive functor, neither is $\Gamma^m$; the functor~$\Gamma^m$ does not preserve shifts or exact triangles in general. However, $\Gamma$ is compatible with direct sums in the following sense: If $P = \bigoplus P^i$ is a (finite) direct sum, then
\[
\Gamma^m(P)\ \cong \bigoplus_{a_i \ge 0,\ \sum a_i = m} \bigotimes_i \Gamma^{a_i}(P^i).
\]
More generally, by \cite[5.4]{Illusie} or \cite[25.2]{LurieSAG}, $\Gamma^* \colonequals \bigoplus_m \Gamma^m$ extends to a monoidal functor on the filtered derived category, which is compatible with the formation of the associated graded object in the above sense. In particular, if $P^\bullet$ is a complex with a finite filtration whose associated graded object is $\bigoplus P^i$, then $\Gamma^m(P^\bullet)$ has a finite filtration with the associated graded object given by
\[
\bigoplus_{a_i \ge 0,\ \sum a_i = m} \bigotimes_i \Gamma^{a_i}(P^i).
\]

In our applications, an ample line bundle $\calO_X(1)$ on $X$ is usually fixed at the outset. Thus, for $E \in D_{coh}(X)$ and any integer $n$, we write $E(n)\colonequals E \otimes_{\calO_X} (\calO_X(1))^{\otimes n}$ as expected.

\section{Proof of the main theorem, and some consequences}

To prove Theorem~\ref{theorem:asymptotic}, we shall need a result which, very roughly speaking, is a variant of Serre vanishing where tensor powers of a sufficiently ample line bundle are replaced by divided powers of a sufficiently ample vector bundle. To make the proof flow better, it is convenient to formulate a more general statement involving complexes as follows:

\begin{proposition}
\label{proposition:gamma:vanishing}
Let $X$ be a projective scheme over a field $\FF$, equipped with an ample line bundle $\calO_X(1)$. Fix a coherent sheaf $F$ and $E \in D^{b}_{coh}(X)\cap D^{\le 0}_{coh}(X)$. Then, for $c \gg 0$, one~has
\[
R\Gamma\big(X,\ \Gamma^m(E(c)\big) \otimes F(l))\ \in\ D^{\le 0}(\FF)
\]
for all integers $l \ge 0$ and $m > 0$.
\end{proposition}

The idea of the proof is to choose a representative of $E$ where each term is a direct sum of twists of the structure sheaf $\calO_X$, and then use Serre vanishing. However, to avoid working with unbounded complexes, we only choose an ``approximate representative'' for~$E$, i.e., one that does not change cohomology in a certain range of degrees. The key point is Lemma~\ref{lemma:gamma:sum:twists}, which ensures that applying derived divided powers to a shift of a ``positive'' complex can only increase ``positivity.''

\begin{proof}
Fix a coherent sheaf $F$ on $X$ as in the statement of the proposition. By Serre vanishing, there exists an integer $j_0 > 0$ such that $H^i(X,\ F(j)) = 0$ for all $i > 0$ and $j \ge j_0$. Stated differently, $R\Gamma(X,\ F(j)) \in D^{\le 0}(\FF)$ for $j \ge j_0$.

For the purpose of the proof, we may replace $E$ by any complex quasi-isomorphic to $E$. By constructing a resolution of $E$ whose terms consist of finite direct sums of twists of $\calO_X$, we may hence assume that $E$ is bounded above by zero, and that each $E^i$ is a finite direct sum of twists of $\calO_X$. Set $d\colonequals\dim X$. For an integer $r$ with $r>d$, set $P^\bullet$ to be
\[
\CD
0 @>>> E^{-r} @>>> E^{-(r-1)} @>>> \cdots @>>> E^{-1} @>>> E^0 @>>> 0.
\endCD
\]
Then each $P^i$ is a finite direct sum of twists of $\calO_X$, and the cokernel $Q^\bullet$ of the injective map~$P^\bullet \to E^\bullet$ lies in $D^{b}_{coh}(X)\cap D^{\le -r}_{coh}(X)$.

For each integer $c$, we view
\[
\varphi\colon P^\bullet(c) \into E^\bullet(c)
\] 
as a one-step decreasing filtration of $E^\bullet(c)$, normalized so that $\mathrm{gr}^1(E^\bullet(c)) = P^\bullet(c)$ and $\mathrm{gr}^0(E^\bullet(c)) = Q^\bullet(c)$. By the compatibility of $\Gamma^m$ with filtrations, as discussed in \S\ref{sec:Prelim}, we obtain an induced filtration on~$\Gamma^m(E^\bullet(c))$ with the associated graded pieces given by
\[
\mathrm{gr}^a(\Gamma^m(E^\bullet(c))\ =\ \Gamma^a(P^\bullet(c)) \otimes \Gamma^b(Q^\bullet(c)), \qquad\text{with } a+b=m,
\]
where negative divided powers are understood to be $0$. Thus, the graded pieces vanish unless $0 \le a \le m$, and $a=0$ gives the ``top'' graded piece (i.e., the quotient) while $a=m$ gives the ``bottom'' graded piece (i.e., a subobject). In particular, the map
\[
\Gamma^m(\varphi)\colon \Gamma^m(P^\bullet(c)) \to \Gamma^m(E^\bullet(c))
\]
identifies with the inclusion 
\[
\CD
\mathrm{gr}^m(\Gamma^m(E^\bullet(c))) @<{\simeq}<< \mathrm{Fil}^m(\Gamma^m(E^\bullet(c))) @>>> \Gamma^m(E^\bullet(c)),
\endCD
\]
and hence its cokernel (which we regard as a representative for its cone in the derived category) carries a filtration whose graded pieces have the form
\[
\Gamma^a(P^\bullet(c)) \otimes \Gamma^b(Q^\bullet(c)), \qquad\text{with } a+b=m \text{ and }b>0.
\]
Since $\Gamma^a$ preserves $D^{\le i}_{coh}(X)$ for $i \le 0$, we have $\Gamma^a(P^\bullet) \in D^{\le 0}_{coh}(X)$ and $\Gamma^b(Q^\bullet) \in D^{\le -d}_{coh}(X)$ provided $b>0$, and hence their tensor product lies in $D^{\le -d}_{coh}(X)$. Since tensoring with $F(j)$ preserves $D^{\le -d}_{coh}(X)$, we see that the cone of
\[
\Gamma^m(P^\bullet(c)) \otimes F(j) \to \Gamma^m(E(c)) \otimes F(j)
\]
also lies in $D^{\le -d}_{coh}(X)$ for all $m \ge 0$ and $c,j \in \ZZ$.

Since $R\Gamma(X,-)$ takes $D^{\le -d}_{coh}(X)$ to $D^{\le 0}(\FF)$, the cone of
\[
R\Gamma(X,\ \Gamma^m(P^\bullet(c)) \otimes F(j)) \to R\Gamma(X,\ \Gamma^m(E(c)) \otimes F(j))
\]
lies in $D^{\le 0}(\FF)$ for all $m \ge 0$ and $c,j \in \ZZ$. It is thus sufficient to prove the proposition when~$E$ is replaced by $P^\bullet$; indeed, for the remainder of the proof, we take $E$ to be $P^\bullet$.

By construction, $P^i = 0$ for $i > 0$ and $i < -r$. Consider the filtration on $P^\bullet(c)$ with the~$i$-th filtered piece given by
\[
\CD
0 @>>> P^{-i}(c) @>>> \cdots @>>> P^{0}(c)@ >>>0.
\endCD
\]
By the compatibility of $\Gamma^m$ with filtrations, we get that $\Gamma^m(P^\bullet(c))$ has a filtration with associated graded object
\[
\bigoplus_{a_i\ge0,\ \sum a_i=m} \Gamma^{a_0}(P^0(c)) \otimes \Gamma^{a_1}(P^{-1}(c)[1]) \otimes \dots \otimes \Gamma^{a_r}(P^{-r}(c)[r])
\]
for each $m \ge 0$ and $c \in \ZZ$. Tensoring with $F(j)$, we see that for each $c,j \in \ZZ$ and $m \ge 0$, the complex $\Gamma^m(P^\bullet(c)) \otimes F(j)$ has a finite filtration with associated graded object
\[
\bigoplus_{a_i\ge0,\ \sum a_i=m} \Gamma^{a_0}(P^0(c)) \otimes \Gamma^{a_1}(P^{-1}(c)[1]) \otimes \dots \otimes \Gamma^{a_r}(P^{-r}(c)[r]) \otimes F(j).
\]
It is thus enough to show: for $m > 0$, $j \ge 0$, and $c \gg 0$, applying $R\Gamma(X,-)$ to each of the terms in the direct sum above produces an object in $D^{\le 0}(\FF)$. Fix such a term corresponding to an index of the form $m = \sum_i a_i$ with $a_i \ge 0$.

As each $P^{-i}$ is a finite direct sum of twists of the structure sheaf, and only finitely many terms $P^{-i}$ are nonzero, we know that for $c \gg 0$, each $P^{-i}(c)$ is a direct sum of line bundles of the form $\calO_X(j)$ for $j \ge j_0$, where $j_0$ was the integer chosen at the start of the proof. By Lemma~\ref{lemma:gamma:sum:twists} below, there are now two possibilities for the term $\Gamma^{a_i}(P^{-i}(c)[i])$ appearing above: if $a_i = 0$, we simply get $\calO_X$, while for $a_i > 0$, we get a complex which is a direct sum of complexes of the form $\calO_X(j) \otimes_\FF V$ with $V \in D^{\le 0}(\FF)$. Since $m=\sum_i a_i$ is positive, we must have $a_i > 0$ for at least one $i$. Thus, the complex displayed above is a direct sum of complexes of the form $F(j) \otimes_\FF V$ for some~$j \ge j_0$ and~$V \in D^{\le 0}(\FF)$. By our choice of~$j_0$, we know that
\[
R\Gamma(X,\ F(j) \otimes_\FF V)\ \in\ D^{\le 0}(\FF)
\]
if $j \ge j_0$ and $V \in D^{\le 0}(\FF)$, which completes the proof.
\end{proof}

\begin{lemma}
\label{lemma:gamma:sum:twists}
Let $X$ be a projective scheme over a field $\FF$, equipped with an ample line bundle $\calO_X(1)$. Let $b,j_1,\dots,j_s$ be integers, where $b \ge 0$, and set
\[
E \colonequals \bigoplus_{i=1}^s \calO_X(j_i)[b],
\]
which is a shift of a direct sum of twists of $\calO_X$. Then, for each integer $a \ge 0$, one has
\[
\Gamma^a(E)\ = \bigoplus_{a_i\ge0,\ \sum a_i=a}\calO_X(a_1j_1 + \dots + a_sj_s)\otimes_\FF\Gamma^{a_1}(\FF[b])\otimes_\FF\dots\otimes_\FF\Gamma^{a_s}(\FF[b]),
\]
where each $\Gamma^{a_i}(\FF[b])$ is a complex of $\FF$-vector spaces lying in $D^{\le 0}(\FF)$.
\end{lemma}

\begin{proof}
As $\Gamma^*(-)$ preserves $D^{\le 0}(\FF)$, the containment in $D^{\le 0}(\FF)$ asserted at the end is automatic. The rest follows from the behavior of $\Gamma^a$ under direct sums, and the fact that
\[
\Gamma^a(\calO_X(j)[b])\ \simeq\ \calO_X(aj)\otimes_\FF\Gamma^a(\FF[b])
\]
for integers $a,b,j$ with $a,b \ge 0$.
\end{proof}

\begin{proof}[Proof of Theorem~\ref{theorem:asymptotic}]
Set $d\colonequals\dim X$, and let $\calI \subset\calO_Y$ be the ideal sheaf of the lci subscheme $X \into Y$, so $\calI/\calI^2$ is the conormal bundle of this closed immersion. Since X is lci and equidimensional, its dualizing complex has the form $\omega_X[d]$ for a line bundle $\omega_X$, so Serre duality says
\[
H^i(X,\ \calO_X(j))\ \cong\ H^{d-i}(X,\ \omega_X(-j))^\vee.
\]
By Serre vanishing, there exists an integer $c_0 \ge 1$ such that
\[
H^{d-i}(X,\ \omega_X(-j))=0\qquad \text{for all $-j\ge c_0$ and $i<d$}.
\]
Equivalently, we have
\[
R\Gamma(X,\ \calO_X(j))\ \in\ D^{\ge d}(\FF) \qquad\text{for } j \le -c_0.
\]
We shall reduce the rest of the proof to the following assertion:

There exists an integer $c_1 \ge 0$ such that, for each integer~$s \ge 1$, one has
\begin{equation}
\label{equation:sym}
R\Gamma\big(X,\ \Sym^s(\calI/\calI^2)(j)\big)\ \in\ D^{\ge d}(\FF) \qquad\text{for } j < -c_1s.
\end{equation}

We claim that~\eqref{equation:sym} implies the theorem. Indeed, given an integer $t \ge 1$ as in the theorem, summing the conclusion of~\eqref{equation:sym} for $s=1,\dots,t-1$ implies that
\[
R\Gamma\big(X_t,\ \calI/\calI^t\big) \in\ D^{\ge d}(\FF)
\]
for $j < -c_1(t-1) = -c_1 t + c_1$, and hence also for $j < -c_1t$. Taking $c=\max(c_0,c_1)$ gives the theorem.

It remains to prove~\eqref{equation:sym}. Let $\calN \colonequals (\calI/\calI^2)^\vee$ denote the normal bundle. Using Serre duality, it suffices to show that there exists $c_1 \ge 0$, such that for each $s \ge 1$, one has
\[
R\Gamma\big(X,\ \Gamma^s(\calN)(j) \otimes \omega_X\big)\ \in\ D^{\le 0}(\FF) \qquad\text{for } j > c_1s.
\]
But this follows from Proposition~\ref{proposition:gamma:vanishing}, since
\[
\Gamma^s(\calN)(as+b)\ =\ \Gamma^s(\calN(a))(b)
\]
for all integers $a,b$.
\end{proof}

We record implications of Theorem~\ref{theorem:asymptotic} for local cohomology modules. By a \emph{standard graded ring} over a field $\FF$, we mean an $\NN$-graded ring $R$ with $R_0=\FF$ that is generated, as an $\FF$-algebra, by finitely many elements of $R_1$. Let $R$ be a standard graded polynomial ring over a field, and let $I$ be a homogeneous ideal. For $t\ge 1$, set $X_t\colonequals\Proj R/I^t$. Let $j$ be an arbitrary integer. Using $\frakm$ to denote the homogeneous maximal ideal of $R$, one has an exact sequence relating local cohomology and sheaf cohomology:
\begin{equation}
\label{equation:four:term}
\CD
0 @>>> {H^0_\frakm(R/I^t)}_j @>>> {(R/I^t)}_j @>>> H^0(X_t,\ \calO_{X_t}(j)) @>>> {H^1_\frakm(R/I^t)}_j @>>> 0.
\endCD
\end{equation}
Moreover, for each $k\ge 1$, one has
\[
H^k(X_t,\ \calO_{X_t}(j))\ =\ {H^{k+1}_\frakm(R/I^t)}_j.
\]

The asymptotic behavior of lengths of local cohomology modules has been studied extensively, see~\cite{Cutkosky} and the references therein. For $R$ an analytically unramified local ring and~$I$ an arbitrary ideal, the limit
\[
\lim_{t\to\infty}\ell(H^0_\frakm(R/I^t))/t^{\dim R}
\]
exists by~\cite[Corollary~6.3]{Cutkosky}. In \cite[Theorem~1.2]{CHST} the authors give an example where this limit is irrational, for $I$ defining a smooth complex projective curve. In the context of local cohomology, Theorem~\ref{theorem:asymptotic} yields the following:

\begin{corollary}
\label{corollary:lim1}
Let $R$ be a standard graded polynomial ring over a field, and $\frakm$ the homogeneous maximal ideal of $R$. Suppose $I$ is a homogeneous ideal such that $R/I$ is equidimensional and $\Proj R/I$ is lci. Then
\[
\limsup_{t\to\infty}\frac{\ell\big(H^k_\frakm(R/I^t)\big)}{t^{\dim R}}\ <\ \infty
\]
for each $k < \dim R/I$.
\end{corollary}

\begin{proof}
The case $k=0$ is covered by~\cite[Corollary~6.3]{Cutkosky}, so assume $k\ge 1$. By Theorem~\ref{theorem:asymptotic} applied to $Y=\PP^n$, with $\calO_Y(1)$ being the standard ample line bundle, there exists an integer~$c \ge 0$, such that for each $t \ge 1$ and $k < \dim R/I$, one has
\[
{H^k_\frakm(R/I^t)}_j = 0 \qquad\text{for } j < -ct.
\]
The result now follows from \cite[Theorem~5.3]{DM:TAMS}.
\end{proof}

\begin{corollary}
\label{corollary:lim2}
Let $R$ be a standard graded polynomial ring over a field, with homogeneous maximal ideal $\frakm$. Suppose $I$ is a homogeneous radical ideal such that $R/I$ is equidimensional and~$\ell\big(H^k_\frakm(R/I^t)\big)<\infty$ for each $k < \dim R/I$ and $t\ge1$. Then, for each $ k <\dim R/I$,
\[
\limsup_{t\to\infty}\frac{\ell\big(H^k_\frakm(R/I^t)\big)}{t^{\dim R}}\ <\ \infty.
\]
\end{corollary}

\begin{proof} For a radical ideal $\fraka$ in a regular local ring $A$, a theorem of Cowsik and Nori implies that~$A/\fraka^t$ is Cohen-Macaulay for each~$t\ge 1$ if and only if $A/\fraka$ is a complete intersection ring, \cite[page~219]{Cowsik-Nori}. The finiteness of the length of each local cohomology module~$H^k_\frakm(R/I^t)$, for $k < \dim R/I$, implies that $(R/I^t)_\frakp$ is Cohen-Macaulay for each~$t\ge 1$ and~$\frakp\in\Spec R\smallsetminus\{\frakm\}$. It follows that $(R/I)_\frakp$ is a complete intersection ring for each~$\frakp\neq\frakm$, and hence that $\Proj R/I$ is lci. The desired result is now immediate from Corollary~\ref{corollary:lim1}.
\end{proof}

\begin{remark}
In the recent paper \cite{DM:MathZ}, the authors prove the following result: Let $R$ be a standard graded ring over a field of characteristic zero; let $\frakm$ denote the homogeneous maximal ideal of $R$. Suppose $I$ is a homogeneous ideal such that $R/I$ is Cohen-Macaulay and of dimension at least $2$, and $I$ is locally a complete intersection on $\Spec R\smallsetminus\{\frakm\}$. Fix an integer $k$ with $k<\dim R/I$. Then, for $t\ge 1$, the lowest degree in which the local cohomology module $H^k_\frakm(R/I^t)$ is nonzero is bounded below by a linear function of $t$.

The hypotheses in \cite{DM:MathZ} are somewhat different from those in Theorem~\ref{theorem:asymptotic} of the present paper, where there is no assumption on the characteristic, nor do we require the ring~$R/I$ to be Cohen-Macaulay.
\end{remark}

\section{Examples}

The following example, which is a variation of \cite[Example~5.7]{BBLSZ}, shows that the bound in Theorem~\ref{theorem:asymptotic} cannot be better than linear; the example also shows that the constant~$c$ in the theorem may be unbounded, even when $\dim X$ is fixed.

\begin{example}
\label{example:linear}
Consider the polynomial ring $R\colonequals\FF[x,y,u,v,w]$, where $\FF$ is a field of arbitrary characteristic. Fix an integer $c\ge2$, and set
\[
I\colonequals (uy-vx,\ vy-wx)+(u,v,w)^c.
\]
The ring $R/I$ has dimension $2$, and the elements $x,y$ form a system of parameters. Since
\[
(R/I)_x=\FF[x,x^{-1},y,u]/(u^c)\quad\text{ and }\quad (R/I)_y=\FF[x,y,y^{-1},w]/(w^c),
\]
one sees that $X\colonequals\Proj R/I$ is lci. We prove that for all integers $t\ge 1$, the asymptotic vanishing in this example takes the form $H^0(X_t,\ \calO_{X_t}(j))=0$ for $j\le -ct$, whereas
\[
H^0(X_t,\ \calO_{X_t}(-ct+1)) \neq 0.
\]
The argument is via local cohomology; the sequence~\eqref{equation:four:term} shows that for $j<0$, one has
\[
H^0(X_t,\ \calO_{X_t}(j))\ =\ {H^1_\frakm(R/I^t)}_j.
\]
We analyze $H^1_\frakm(R/I^t)$ using the \v Cech complex
\[
\CD
0 @>>> R/I^t @>>> {(R/I^t)}_x\oplus {(R/I^t)}_y @>>> {(R/I^t)}_{xy} @>>> 0,
\endCD
\]
and claim that
\begin{equation}
\label{equation:cocycle}
\left[\left(\frac{u}{x^2}\right)^{ct-1},\ \left(\frac{w}{y^2}\right)^{ct-1}\right]\ \in\ {(R/I^t)}_x\oplus {(R/I^t)}_y
\end{equation}
determines a nonzero element of ${H^1_\frakm(R/I^t)}_{-ct+1}$. To verify that the displayed element is indeed a \v Cech cocycle, it suffices to verify that
\[
(uy^2)^{ct-1}-(wx^2)^{ct-1}\ \in\ I^t.
\]
Since the ideal $I$ contains $uy^2-wx^2$ as well as ${(uy^2)}^c$, it suffices to check that
\[
(uy^2)^{ct-1}-(wx^2)^{ct-1}\ \in\ \Big(uy^2-wx^2,\ {(uy^2)}^c\Big)^t
\]
in the polynomial ring $\FF[x,y,u,v,w]$, and hence in its subring $\FF[uy^2,\, wx^2]$. Setting $a\colonequals uy^2$ and $b\colonequals wx^2$ for notational simplicity, it suffices to check that
\[
a^{ct-1}-b^{ct-1}\ \in\ \big(a-b,\ a^c\big)^t
\]
in the polynomial ring $\FF[a,b]$. Replacing $b$ by $a-b$, we need to show
\[
a^{ct-1}-(a-b)^{ct-1}\ \in\ \big(b,\ a^c\big)^t,
\]
which is evident by considering the binomial expansion of $(a-b)^{ct-1}$. This completes the argument that~\eqref{equation:cocycle} is indeed a \v Cech cocycle.

To verify that ${(u/x^2)}^{ct-1}$ is nonzero in ${(R/I^t)}_x$, note that its image under the surjection
\[
{(R/I^t)}_x \onto {\left({\frac{R}{(uy-vx,\ vy-wx)+(u,v,w)^{ct}}}\right)}_x=\FF[x,x^{-1},y,u]/(u^{ct})
\]
is nonzero. As it has negative degree, the element~\eqref{equation:cocycle} cannot be in the image of
\[
R/I^t \to {(R/I^t)}_x\oplus {(R/I^t)}_y,
\]
which completes the argument that
\[
H^0(X_t,\ \calO_{X_t}(-ct+1))\ =\ {H^1_\frakm(R/I^t)}_{-ct+1}\ \neq\ 0.
\]

Next, we examine the intersection of ${(R/I^t)}_x$ and ${(R/I^t)}_y$ in ${(R/I^t)}_{xy}$. For this, consider the $\ZZ^3$-grading with
\begin{alignat*}3
\deg u &= (2,0,-1), \qquad\qquad & \deg x &= (1,0,0),\\
\deg v &= (1,1,-1), \qquad\qquad & \deg y &= (0,1,0),\\
\deg w &= (0,2,-1).
\end{alignat*}
Each homogeneous element of ${(R/I^t)}_x$ has degree $(i,j,k)$ with $j\ge 0$ and~$k>-ct$, whereas, in ${(R/I^t)}_y$, each homogeneous element has degree $(i,j,k)$ with $i\ge 0$ and~$k>-ct$. Thus, a homogeneous element in the intersection must have degree $(i,j,k)$ satisfying~$i\ge 0$,~$j\ge 0$, and~$k>-ct$. But the $\ZZ^3$-grading specializes to the standard $\NN$-grading on $R$ under the map
\[
\ZZ^3 \to \ZZ\qquad\text{with}\qquad(i,j,k)\mapsto i+j+k,
\]
implying that each homogeneous element in the kernel of
\[
{(R/I^t)}_x\oplus {(R/I^t)}_y \to {(R/I^t)}_{xy}
\]
has degree greater than $-ct$. It follows that
\[
H^0(X_t,\ \calO_{X_t}(j))\ =\ {H^1_\frakm(R/I^t)}_{j}\ =\ 0 \qquad\text{for } j \le -ct.
\]
\end{example}

Theorem~\ref{theorem:asymptotic} may fail if $X$ is not lci:

\begin{example}
\label{example:non:lci}
Let $Z$ denote the Segre embedding of $\PP^1\times\PP^2$ in~$\PP^5$, over a field $\FF$ of characteristic zero, and set $X\subset\PP^6$ to be the projective cone over $Z$. Then $X$ has dimension $4$, and is Cohen-Macaulay though not lci. If $t\ge 2$, we claim that
\[
H^3(X_t,\ \calO_{X_t}(j)) \neq 0 \qquad\text{for each } j<0.
\]

By \cite[Example~5.1]{BBLSZ}, if $t\ge 2$, then $H^2(Z_t,\ \calO_{Z_t})\neq 0$, i.e., ${H^3_{\frakm_R}(R/I^t)}_0\neq 0$, where~$R/I$ is the homogeneous coordinate ring for $Z\subset\PP^5$. But then $X\subset\PP^6$ has homogeneous coordinate ring $S/IS$, where $S\colonequals R[y]$ with $y$ being a new indeterminate, so
\[
H^4_{\frakm_S}(S/I^tS)\ \cong\ H^3_{\frakm_R}(R/I^t)\otimes_\FF H^1_{(y)}(\FF[y])
\]
has a nonzero graded component in each negative degree, which proves the claim.
\end{example}

Lastly, Theorem~\ref{theorem:asymptotic} may fail if $X$ is lci but not equidimensional:

\begin{example}
\label{example:non:equidimensional}
Consider the polynomial ring $R\colonequals\FF[x,y,z]$, where $\FF$ is a field of arbitrary characteristic, and set $I\colonequals (xy,\, xz)$. Then $R/I$ has dimension $2$, and $X\colonequals\Proj R/I$ is smooth, hence lci. Fix $t\ge1$. The exact sequence
\[
\CD
0 @>>> R/I^t @>>> R/(x^t)\oplus R/(y,z)^t @>>> R/(x^t+(y,z)^t) @>>> 0
\endCD
\]
induces an isomorphism
\[
{H^1_\frakm(R/I^t)}_j\ \cong\ {H^1_\frakm(R/(y,z)^t)}_j \qquad\text{for } j<0
\]
which shows that $H^1_\frakm(R/I^t)$ has a nonzero graded component in each negative degree, so
\[
H^0(X_t,\ \calO_{X_t}(j))\ \neq\ 0 \qquad\text{for each } j<0.
\]
\end{example}


\end{document}